\documentclass[11pt]{article}
\usepackage[margin=1in]{geometry}
\usepackage{amsmath,amssymb,amsthm,mathtools}
\usepackage{booktabs,array}
\usepackage{enumitem}
\usepackage{microtype}
\usepackage[hidelinks]{hyperref}
\usepackage[T1]{fontenc}
\usepackage{lmodern}
\allowdisplaybreaks
\newtheorem{theorem}{Theorem}[section]
\newtheorem{proposition}[theorem]{Proposition}
\newtheorem{lemma}[theorem]{Lemma}

\theoremstyle{definition}

\newcommand{\per}{\operatorname{per}}
\newcommand{\Tr}{\operatorname{Tr}}

\newcommand{\C}{\mathbb C}

\newcommand{\dbar}{\overline d}
\newcommand{\youngcovers}{\succ}
\newcommand{\youngcovered}{\prec}

\title{Lieb's Permanental Dominance Conjecture for Ordinary Immanants through Order Fifteen}
\author{YINJIE LI\textsuperscript{a}\\[-0.1em]
\small\textsuperscript{a}\,Independent Researcher, Nanjing, China\\[-0.1em]
\small Email: \texttt{stewarteur@gmail.com}}
\date{}

\begin{document}
\maketitle

\begin{abstract}
Pate proved ordinary irreducible-immanant permanental dominance through order $13$ and identified $(4,4,3,3)$ as the sole remaining order-$14$ case, with $(5,4,3,3)$ and $(3^5)$ forming the order-$15$ frontier.  These three cases are settled here; consequently $d_\lambda(A)/f^\lambda\le\per(A)$ for every partition $\lambda\vdash n$ with $n\le15$ and every complex Hermitian positive-semidefinite matrix $A$.  The argument also yields results beyond this finite frontier: an exact four-term bridge for $(4,4,3,3)$, the uniform family $(m,4,3,3)$, a two-parameter family $(a,b,3,3)$ for $a\ge b\ge4$ and $5a\ge8b$, and a long-first-row criterion for arbitrary fixed tails.  These results arise from explicit specializations of Pate's $W$-function positivity framework using partial swaps, Young projectors, Pieri--content identities, and branching data.  For $(3^5)$, an exact Farkas certificate shows that the central-projector partial-swap cone is insufficient; a branching-refined one-swap construction escapes this obstruction and yields a positive $106+19$-witness certificate.  Boundary-compression and node-moving results further describe the reach and limitations of the local-filter method.  All finite certificates are checked by exact integer or rational arithmetic and are supplied as ancillary material.  The order-$14$ bridge is additionally formalized and kernel-checked in Lean~4 for all complex Hermitian positive-semidefinite matrices, including the exact coefficient normalization and the deduction of $(4,4,3,3)$ permanental dominance from four explicitly stated Pate inequalities.
\end{abstract}

\begingroup
\renewcommand{\thefootnote}{}
\footnotetext{\textit{2020 Mathematics Subject Classification.} 15A15, 15A45, 15A69, 20C30.\\
\textit{Key words and phrases.} permanent, immanant, positive-semidefinite matrix, symmetric group, Young diagram, Jucys--Murphy element, exact certificate.}
\addtocounter{footnote}{-1}
\endgroup

\section{Introduction and main results}
For a partition $\lambda\vdash n$ with irreducible character $\chi^\lambda$ of $S_n$, define
\begin{equation}\label{eq:immanant}
 d_\lambda(A)=\sum_{\sigma\in S_n}\chi^\lambda(\sigma)\prod_{i=1}^n a_{i,\sigma(i)},
 \qquad \dbar_\lambda(A)=\frac{d_\lambda(A)}{f^\lambda},
\end{equation}
where $f^\lambda=\chi^\lambda(1)$.  We use no $n!$ normalization in $d_\lambda$.  The permanent corresponds to the one-row partition $(n)$.  The ordinary symmetric-group specialization of Lieb's permanental dominance conjecture \cite{Lieb1966} asks whether
\begin{equation}\label{eq:pdc}
 \dbar_\lambda(A)\le \per(A)
\end{equation}
for every complex Hermitian positive-semidefinite $n\times n$ matrix $A$.  We abbreviate permanental dominance by PDC and occasionally suppress commas and parentheses in partition subscripts.  We write $\nu\youngcovers\beta$ (equivalently $\beta\youngcovered\nu$) when $\nu$ is obtained from $\beta$ by adding one box.

The distinction between \eqref{eq:pdc} and the full subgroup form of Lieb's conjecture is essential.  Nothing in this paper proves the latter.  Our result closes only the ordinary irreducible-immanant frontier through order fifteen.

Pate's work in the 1990s provides both the historical frontier and the positivity architecture behind our witnesses.  His 1997 paper develops a general machine based on local group-algebra filters, products of disjoint transpositions crossing two tensor blocks, positivity on decomposable tensors, and central averaging \cite{Pate1997}.  His 1999 paper states a broad family of permanent-dominance results, proves all ordinary irreducible-immanant cases for $n\le13$, and records that $(4,4,3,3)$ is the only remaining order-$14$ shape while at order $15$ only $(5,4,3,3)$ and $(3^5)$ remain \cite{Pate1999}.  Pate remarks that $(3^5)$ would probably require a new tensor inequality.  We therefore do not claim as new the general $f^*wf$ mechanism, cross-block partial-swap positivity, central averaging, or the classical content calculus of Young's seminormal theory.  The new content here lies in explicit computable specializations, exact bridge certificates, uniform families, a branching refinement, and structural results about the resulting cones.

Our first finite frontier theorem is the following exact bridge.

\begin{theorem}[Order-$14$ bridge]\label{thm:bridge14}
For every complex Hermitian positive-semidefinite $A\in M_{14}(\C)$,
\begin{align}
59512\,d_{(4,4,3,3)}(A)\le{}&4035\,d_{(6,5,3)}(A)+6725\,d_{(6,4,4)}(A)\notag\\
&+39759\,d_{(5,5,4)}(A)+23636\,d_{(5,3,3,3)}(A).\label{eq:bridge14}
\end{align}
Since
\[
(f^{(4,4,3,3)},f^{(6,5,3)},f^{(6,4,4)},f^{(5,5,4)},f^{(5,3,3,3)})
=(12012,15015,9009,6006,15015),
\]
this is equivalently
\begin{equation}\label{eq:bridge14norm}
\dbar_{4433}\le
\frac{20175}{238048}\dbar_{653}+
\frac{20175}{238048}\dbar_{644}+
\frac{79518}{238048}\dbar_{554}+
\frac{118180}{238048}\dbar_{5333}.
\end{equation}
The four coefficients in \eqref{eq:bridge14norm} are positive and sum to one.
\end{theorem}

Every right-hand partition in \eqref{eq:bridge14norm} belongs to Pate's established classes, so Theorem~\ref{thm:bridge14} proves PDC for $(4,4,3,3)$ and hence for every ordinary immanant of order at most fourteen.

The bridge is not an isolated calculation.  It led us to an explicit family of partial-swap witnesses.  When the first local block is symmetric, their irreducible coefficients admit a finite-difference formula in the contents of the added boxes.  This yields the following uniform statements.

\begin{theorem}[Uniform four-row family]\label{thm:m433}
For every integer $m\ge4$ and every complex Hermitian positive-semidefinite matrix of order $m+10$,
\[
 \dbar_{(m,4,3,3)}(A)\le \per(A).
\]
\end{theorem}

For $m\ge6$ this follows from an all-parameter three-witness identity; $m=5$ is an exact seven-witness order-$15$ bridge, and $m=4$ is Theorem~\ref{thm:bridge14}.

The same content formula yields a substantially broader theorem.  If $b=(i,j)$ is a box, write $c(b)=j-i$ and, for a partition $\rho$, set
\[
 C(\rho)=\max\{\rho_i-i:\rho_i>\rho_{i+1}\},
\]
the largest content of a removable box.

\begin{theorem}[Arbitrary-tail long-first-row theorem]\label{thm:longrow}
Let $\rho\vdash q$ be nonempty.  If
\[
 m\ge \rho_1,\qquad m\ge q+C(\rho)-1,
\]
then PDC holds for $(m,\rho)$.  More precisely, if $\tau\subset\rho$, $\rho/\tau$ is a nonempty horizontal $j$-strip, and
\[
 \nu_\tau=(m+j,\tau),\qquad
 R_{m,\rho}(\tau)=jm+\binom j2-\sum_{b\in\rho/\tau}c(b)-q+j,
\]
then $R_{m,\rho}(\tau)\ge0$ and
\begin{equation}\label{eq:longrowrec}
 \dbar_{(m,\rho)}\le
 \sum_{\tau\ne\rho}
 \frac{R_{m,\rho}(\tau)f^{(m+j,\tau)}}{qf^{(m,\rho)}}
 \dbar_{(m+j,\tau)}.
\end{equation}
The coefficients in \eqref{eq:longrowrec} sum to one.  Each right-hand tail is strictly smaller, and each nonempty right-hand tail satisfies the same criterion.
\end{theorem}

A second explicit combination enlarges the balanced four-row region.

\begin{theorem}[Two-parameter family]\label{thm:ab33}
If $a\ge b\ge4$ and $5a\ge8b$, then
\[
 \dbar_{(a,b,3,3)}(A)\le \per(A)
\]
for every complex Hermitian positive-semidefinite $A$ of the appropriate order.
\end{theorem}

The coefficient signs in Theorem~\ref{thm:ab33} are certified for the entire infinite parameter region by exact Bernstein expansions and translated nonnegative polynomial checks.  The complete rational certificate is ancillary; no floating-point inference enters the proof.

At order fifteen, however, the central-projector partial-swap cone meets an exact obstruction.

\begin{theorem}[Central-cone obstruction]\label{thm:separator}
Let $C_{15}$ be the cone of normalized-immanant coefficient vectors generated by the central-projector two-block partial-swap witnesses of all block sizes and all swap sizes.  Write $\mathbf e_\nu$ for the standard coordinate vector.  Then
\[
 \mathbf e_{(15)}-\mathbf e_{(3^5)}\notin C_{15}.
\]
More strongly, an exact integer separating functional is nonnegative on every generator of $C_{15}$, on the individual immanant-nonnegativity rays, on Schur lower-bound rays, and on every other order-$15$ PDC ray, while it is negative on the target direction.
\end{theorem}

The obstruction motivates retaining one additional branching step inside the large local block.  For $\gamma\youngcovered\beta$ define $c(\beta/\gamma)$ to be the content of the added box.

\begin{theorem}[Branching-refined one-swap witness]\label{thm:branching}
Let $\beta\vdash n-1$ and $\gamma\youngcovered\beta$, $\gamma\vdash n-2$.  Then
\begin{equation}\label{eq:branching}
 B_{\beta,\gamma}(A)=
 \sum_{\nu\youngcovers\beta}
 \frac{d_\nu(A)}{c(\nu/\beta)-c(\beta/\gamma)}\ge0
\end{equation}
for every complex Hermitian positive-semidefinite $A$.  All denominators are nonzero.
\end{theorem}

This refinement is enough to close the rectangle.

\begin{theorem}[The order-$15$ rectangle]\label{thm:rect15}
For every complex Hermitian positive-semidefinite $15\times15$ matrix $A$,
\[
 d_{(3,3,3,3,3)}(A)\le6006\,\per(A),
 \qquad\text{equivalently}\qquad
 \dbar_{(3^5)}(A)\le\per(A).
\]
There are $125$ explicitly defined nonnegative witnesses $R_i$, consisting of $106$ central-projector witnesses and $19$ branching-refined witnesses, and positive integers $z_i,M$ such that
\begin{equation}\label{eq:rectcert}
 \sum_{i=1}^{125}z_iR_i
 =M\bigl(6006\mathbf e_{(15)}-\mathbf e_{(3^5)}\bigr).
\end{equation}
Here the $R_i$ are represented by their raw-immanant coefficient vectors, and $\mathbf e_\nu$ denotes the standard coordinate vector.  The identity is verified exactly in all $p(15)=176$ partition coordinates.
\end{theorem}

Combining Pate's frontier theorem with Theorems~\ref{thm:bridge14}, \ref{thm:m433}, and \ref{thm:rect15} gives the headline conclusion.

\begin{theorem}[Permanental dominance through order fifteen]\label{thm:through15}
For every $n\le15$, every partition $\lambda\vdash n$, and every complex Hermitian positive-semidefinite $n\times n$ matrix $A$,
\[
 \dbar_\lambda(A)\le\per(A).
\]
This is the ordinary irreducible-immanant statement for $S_n$; it is not the full subgroup form of Lieb's conjecture.
\end{theorem}

The final part of the paper studies the architecture rather than a fixed order.  At swap size $k$, arbitrary local filters can be compressed to the last $k$ Young-branching steps, with multiplicity spaces of dimension at most $k!$.  At $k=1$ and block sizes $(n-1,1)$ this compression is exact enough to show saturation: arbitrary local filters generate precisely the cone of Theorem~\ref{thm:branching}.  Thus retaining longer Young paths cannot strengthen that architecture.  We then determine the exact node-moving closure of the explicit theorem system and prove that its covered partitions have asymptotic density zero.  This density statement concerns only the specified family-and-node-moving program and is not a no-go theorem for unrestricted $W$-functions or higher-swap coherent filters.

\section{Background: immanants, Pate's order, and tensor witnesses}\label{sec:background}
\subsection{Gram tensors and central projectors}
Let $A\succeq0$ be complex Hermitian.  Choose Gram vectors $x_1,\dots,x_n$ in a finite-dimensional complex Hilbert space $H$ so that $a_{ij}=\langle x_i,x_j\rangle$, and put $x=x_1\otimes\cdots\otimes x_n$.  Let $U_\sigma$ permute tensor positions with the convention
\[
 \langle x,U_\sigma x\rangle=\prod_i a_{i,\sigma(i)}.
\]
For $\lambda\vdash n$, the central orthogonal idempotent
\[
 e_\lambda=\frac{f^\lambda}{n!}\sum_{\sigma\in S_n}\chi^\lambda(\sigma^{-1})\sigma
\]
satisfies
\begin{equation}\label{eq:projpairing}
 \langle x,U(e_\lambda)x\rangle=\frac{f^\lambda}{n!}d_\lambda(A).
\end{equation}
This realizes ordinary immanants as isotypic tensor norms and immediately covers singular PSD matrices: no invertibility hypothesis is present in the Gram construction.

\subsection{Pate's positivity architecture and historical frontier}
Pate's 1997 ``machine'' is built from two tensor blocks, local group-algebra filters, and an involution that is a product of disjoint transpositions crossing the blocks \cite{Pate1997}.  In present notation, its core pattern is
\[
 W=(a\otimes b)^*\tau(a\otimes b),
\]
with $\tau$ a cross-block involution.  The associated generalized matrix function is nonnegative on positive-semidefinite Hermitian matrices, and averaging conjugates produces a class function and hence an immanant inequality.  The present paper uses this architecture rather than claiming it anew.

Pate's 1999 paper gives a particularly convenient statement of the known ordinary-immanant frontier \cite{Pate1999}.  It proves PDC for partitions of the form $(p,q^u,r^v,2^s,1^t)$ with $0\le u\le2$ and $0\le v\le1$, and deduces PDC for all ordinary immanants of order at most $13$.  It then states explicitly that $(4,4,3,3)$ is the sole remaining order-$14$ case and that at order $15$ only $(5,4,3,3)$ and $(3^5)$ remain.  We use that statement as our historical baseline.

The node-moving comparison used later is also not new.  Pate's 1997 paper restates the earlier node-diagram theorem: moving an outer node from a row of length greater than one to a new singleton row weakly decreases the normalized immanant \cite{Pate1997,Pate1992}.  We rely on the inspected 1997 restatement; the 1992 bibliographic source is listed for provenance.

\section{The order-fourteen bridge}\label{sec:order14}
The order-$14$ frontier can be closed by a positive linear certificate.  We summarize the conceptual construction and leave the twenty rows of exact coefficient data in the ancillary archive.

Split the tensor factors so that $k\in\{1,2\}$ positions in one local block are paired with $k$ positions in the other.  Let $P,Q$ be suitable Young projectors on nested initial segments and, for $k=2$, let $B$ be the symmetric or alternating projector on the final pair.  If $\tau_k$ is the product of the paired transpositions, the filtered operator has the form
\[
 W=QBP\,U_{\tau_k}BQ.
\]
The full construction in \texttt{anc/order14/PROOF.md}, Sections 2--4, obtains this operator via partial transpose from a positive operator, showing that $\langle x,Wx\rangle\ge0$ on every decomposable Gram tensor.  Central averaging therefore yields a linear immanant inequality.  Twenty explicitly selected witnesses are combined with strictly positive rational weights.  Exact cancellation on all $p(14)=135$ partitions leaves precisely the five coordinates in \eqref{eq:bridge14}.

\begin{proposition}[Exact certificate identity]\label{prop:cert14}
There are twenty explicitly specified universally nonnegative witness forms $W_j$ and positive rational weights $q_j$ such that their central immanant coefficient vectors satisfy
\[
 \sum_{j=1}^{20}q_jW_j
 =4035d_{653}+6725d_{644}+39759d_{554}+23636d_{5333}-59512d_{4433}.
\]
The identity is an equality over $\mathbb Q$ in all $135$ irreducible coordinates.
\end{proposition}

The hook-length degrees stated in Theorem~\ref{thm:bridge14} give the normalized weights.  A useful consistency check is
\[
 59512\cdot12012=4035\cdot15015+6725\cdot9009+39759\cdot6006+23636\cdot15015=714858144.
\]
Thus \eqref{eq:bridge14norm} is exactly convex rather than approximately so.

The final Lean project included with the ancillary material formalizes the new bridge over complex Hermitian PSD matrices, including the tensor/Gram and representation-theoretic links, and formalizes the normalized transfer to $(4,4,3,3)$ conditional on four theorem parameters representing the historical Pate cases.  A separate final validation rebuilt the pinned Lean~4.19.0 project and the full \texttt{Bridge.Young} proof chain successfully.  The project contains no \texttt{sorry} or \texttt{admit}; a theorem-dependency audit found no \texttt{sorryAx} and no additional mathematical axioms beyond \texttt{propext}, \texttt{Classical.choice}, and \texttt{Quot.sound}.  The project does not formalize the historical Pate proofs or the order-$15$ results.

\section{Explicit partial-swap and Pieri--content witnesses}\label{sec:pieri}
\subsection{General local filters}
Let $n=p+q$, let $a\in\C[S_p]$ and $b\in\C[S_q]$ be arbitrary, and for $1\le k\le\min(p,q)$ put
\[
 \tau_k=(1,p+1)\cdots(k,p+k),\qquad
 W=(a\otimes b)^*\tau_k(a\otimes b).
\]
For $\nu\vdash n$ set
\[
 t_\nu=\Tr_{S^\nu}\rho_\nu(W).
\]

\begin{proposition}[Pate-type partial-swap positivity]\label{prop:partialswap}
For every complex Hermitian PSD matrix $A$,
\begin{equation}\label{eq:partialswap}
 \sum_{\nu\vdash n}t_\nu d_\nu(A)\ge0.
\end{equation}
\end{proposition}

\begin{proof}
After applying the local filters, the decomposable tensor separates as $u\otimes v$ across the two blocks.  Resolve the $k$ swapped tensor positions and the unswapped remainders, say $u_{I,R}$ and $v_{I,S}$.  Direct expansion yields
\[
 \langle u\otimes v,U(\tau_k)(u\otimes v)\rangle
 =\sum_{R,S}\left|\sum_Iu_{I,R}\overline{v_{I,S}}\right|^2\ge0.
\]
Equivalently this is the trace of a product of two reduced PSD operators.  Conjugating $W$ by a permutation preserves nonnegativity on decomposable tensors.  Central averaging and Schur's lemma give
\[
 Z(W)=\frac1{n!}\sum_{g\in S_n}gWg^{-1}
 =\sum_\nu\frac{t_\nu}{f^\nu}e_\nu.
\]
Pairing with a Gram tensor and using \eqref{eq:projpairing} gives \eqref{eq:partialswap}.  This proof neither assumes nor concludes that $W$ is PSD on the full tensor space.
\end{proof}

For central projectors $a=e_\alpha$, $b=e_\beta$, cyclicity gives the exact character sum
\begin{equation}\label{eq:centralcoeff}
 t_\nu=\frac{f^\alpha f^\beta}{p!q!}
 \sum_{\sigma\in S_p,\pi\in S_q}
 \chi^\alpha(\sigma)\chi^\beta(\pi)
 \chi^\nu((\sigma\oplus\pi)\tau_k).
\end{equation}
For $k\ge1$ the witness is balanced:
\begin{equation}\label{eq:balance}
 \sum_\nu f^\nu t_\nu=0.
\end{equation}
The balance is the regular-representation trace: the cross-block permutation has zero identity coefficient in the block subgroup.

\subsection{The symmetric-block content formula}
Suppose the first block carries the symmetric representation $(p)$ and the second block carries $\beta\vdash q$.  For a horizontal $p$-strip $\nu/\beta$ define
\[
 P_{\nu/\beta}(z)=\prod_{b\in\nu/\beta}(z+c(b)),
\]
and set
\begin{equation}\label{eq:Ndef}
 N^\nu_{p,\beta,k}=\frac1{(p-k)!}
 \sum_{j=0}^{p-k}(-1)^{p-k-j}\binom{p-k}{j}
 P_{\nu/\beta}(j-p+1),
\end{equation}
with $N=0$ when $\nu/\beta$ is not a horizontal $p$-strip.

\begin{theorem}[Pieri--content finite differences]\label{thm:pieri}
For every $\beta\vdash q$ and $1\le k\le\min(p,q)$,
\begin{equation}\label{eq:phi}
 \Phi_{p,\beta,k}(A)=\sum_{\nu\vdash p+q}N^\nu_{p,\beta,k}d_\nu(A)\ge0,
\end{equation}
and
\[
 \sum_\nu f^\nu N^\nu_{p,\beta,k}=0.
\]
\end{theorem}

\begin{proof}
Pieri's rule makes the $(p)\otimes\beta$ subgroup type multiplicity one precisely for horizontal $p$-strips.  Order the old labels before the new labels and let $J_i=\sum_{j<i}(j,i)$ be the Jucys--Murphy elements.  On the selected subgroup type,
\[
 \prod_{r=1}^p(z+J_{q+r})
\]
acts by the scalar $P_{\nu/\beta}(z)$.  Expanding this product by inserting the new labels into permutation cycles groups the terms according to the number $k$ of old labels lying in nontrivial mixed cycles.  Modulo permutations of the symmetric new block, a term with $k$ such old labels reduces to $k$ disjoint cross-block transpositions.  Counting the insertion patterns yields the Newton expansion
\[
 P_{\nu/\beta}(z)
 =\sum_{k=0}^pN^\nu_{p,\beta,k}(z+k)^{\overline{p-k}},
\]
where $(z)^{\overline r}=\prod_{j=0}^{r-1}(z+j)$ is a rising factorial.  After tracing, one obtains
\[
 \Tr_{S^\nu}\bigl(\tau_k(e_{(p)}\otimes e_\beta)\bigr)
 =\frac{f^\beta}{\binom pk(q)_k}N^\nu_{p,\beta,k}.
\]
Here $(q)_k=\prod_{j=0}^{k-1}(q-j)$ is a falling factorial.  The prefactor is positive and independent of $\nu$, so Proposition~\ref{prop:partialswap} proves \eqref{eq:phi}; \eqref{eq:balance} proves the balance relation.
\end{proof}

On the horizontal-strip support, the $k=1$ formula simplifies to
\begin{equation}\label{eq:k1content}
 N^\nu_{p,\beta,1}=\kappa_\nu-\kappa_\beta-\binom p2,
\qquad \kappa_\lambda=\sum_{b\in\lambda}c(b).
\end{equation}
For $k=2$ one obtains a quadratic expression in total and squared contents, while for $p=4,k=3$,
\[
 N^\nu_{4,\beta,3}=P_{\nu/\beta}(-2)-P_{\nu/\beta}(-3).
\]
The zeros of the content polynomial create strong sparsity.  This is the computational feature that makes exact symbolic search practical.

\section{Uniform dominance families}\label{sec:families}
\subsection{\texorpdfstring{The family $(m,4,3,3)$}{The family (m,4,3,3)}}
For $m\ge6$ define
\[
 D_m=\Phi_{4,(m,3,3),3},\qquad
 E_m=\Phi_{m,(4,3,3),1},\qquad
 F_m=\Phi_{m+1,(3,3,3),2}.
\]
The combination
\begin{equation}\label{eq:threewitness}
 9D_m+90(m-3)E_m+5(m-3)F_m
\end{equation}
is a positive combination of nonnegative witnesses.  Its coefficient at $(m,4,3,3)$ is $-180(5m-9)$; the remaining seven coefficients are
\begin{center}\small
\begin{tabular}{ll}
\toprule
partition $\mu_i$ & coefficient $r_i(m)$\\
\midrule
$(m+4,3,3)$ & $3m(17m^2+74m-351)$\\
$(m+3,4,3)$ & $9(m-3)(m^2+27m-8)$\\
$(m+3,3,3,1)$ & $14m^3+21m^2-711m+1296$\\
$(m+2,4,3,1)$ & $18(m-3)(7m-24)$\\
$(m+2,3,3,2)$ & $10(5m^2-41m+114)$\\
$(m+1,4,3,2)$ & $90(m-3)(m-6)$\\
$(m+1,3,3,3)$ & $90(m-1)(m-6)$\\
\bottomrule
\end{tabular}
\end{center}
Writing $m=t+6$ makes each polynomial manifestly nonnegative.  Balance gives
\[
 \sum_i r_i(m)f^{\mu_i}=180(5m-9)f^{(m,4,3,3)},
\]
so the resulting normalized inequality is convex.  The first six right-hand partitions have at most three parts exceeding $2$, and the last is covered by Pate's repeated-row class.  This proves Theorem~\ref{thm:m433} for $m\ge6$.

The remaining endpoint $m=5$ is the exact identity
\begin{align}\label{eq:5433bridge}
\dbar_{(5,4,3,3)}\le{}&\frac{14}{33}\dbar_{(9,3,3)}+
\frac{791}{1485}\dbar_{(8,4,3)}\\
&+\frac{2}{55}\dbar_{(7,5,3)}+
\frac{2}{297}\dbar_{(7,4,4)}.\notag
\end{align}
It is obtained from seven explicitly specified $\Phi_{p,\beta,k}$ witnesses and checked by exact integer arithmetic.  All four right-hand shapes have three rows.  The $m=4$ endpoint is the order-$14$ bridge.

For completeness, Pate's node-moving comparison applied to the first row of $(5,4,3,3)$ also gives
\begin{equation}\label{eq:44331}
 \dbar_{(4,4,3,3,1)}\le\dbar_{(5,4,3,3)}\le\per.
\end{equation}
This explicit treatment is useful when checking closure by node movement, although Pate's 1999 frontier statement already isolates the two order-$15$ shapes $(5,4,3,3)$ and $(3^5)$.

\subsection{Proof of the arbitrary-tail recurrence}
We prove Theorem~\ref{thm:longrow}.  Apply the $k=1$ witness \eqref{eq:k1content} with first block length $m$ and tail $\rho$.  By Pieri interlacing, every horizontal $m$-strip child can be written uniquely as
\[
 \nu=(m+j,\tau),
\]
where $\rho/\tau$ is a horizontal $j$-strip.  The $j=0$ child is the target $(m,\rho)$.  Shifting every row of $\tau$ down by one changes total content by $-|\tau|$, and one obtains
\[
 \kappa_{(m+j,\tau)}=\binom{m+j}{2}+\kappa_\tau-|\tau|.
\]
Substitution into \eqref{eq:k1content} gives the target coefficient $-q$ and the off-target coefficient $R_{m,\rho}(\tau)$ in Theorem~\ref{thm:longrow}.

If $j$ boxes are removed from $\rho$ in a horizontal strip, each has content at most $C(\rho)$, hence
\[
 R_{m,\rho}(\tau)\ge j(q-1)+\binom j2-q+j=(j-1)q+\binom j2\ge0.
\]
The witness inequality therefore gives
\[
 qd_{(m,\rho)}\le\sum_{\tau\ne\rho}R_{m,\rho}(\tau)d_{(m+j,\tau)}.
\]
After dividing by $qf^{(m,\rho)}$, balance gives the exact coefficient sum one in \eqref{eq:longrowrec}.

It remains to close the induction.  Removing one outer corner can create a new maximal removable content larger by at most one, so after removing $j$ boxes
\[
 C(\tau)\le C(\rho)+j.
\]
Therefore
\[
 |\tau|+C(\tau)-1\le q+C(\rho)-1\le m<m+j,
\]
and $m+j\ge\tau_1$.  Every nonempty right-hand tail again satisfies the theorem's criterion and is strictly smaller.  Strong induction on $q$ terminates at the permanent.

The threshold is exact for this single-witness argument: removing a box of content $C(\rho)$ gives a $j=1$ child with coefficient $m-q-C(\rho)+1$, which is negative below the stated threshold.  This is not a no-go theorem for positive combinations of several witnesses.

\subsection{\texorpdfstring{The region $(a,b,3,3)$}{The region (a,b,3,3)}}
When $a\ge2b+4$, Theorem~\ref{thm:longrow} already applies.  In the remaining interval define
\[
 D=\Phi_{b,(a,3,3),3},\qquad
 E=\Phi_{a,(b,3,3),1},
\]
and
\begin{equation}\label{eq:abcombination}
 W=(2b+4-a)D+b(b+1)(3a-4b+4)E.
\end{equation}
The possible support can be indexed by
\[
 \nu_{g,c}=(a+b-g-c,g+3,3,c),
 \qquad0\le g\le b-3,\quad0\le c\le3,
\]
omitting a final zero.  Content-polynomial zeros force all larger $g$ coordinates to vanish.  Let $H_{g,c}(a,b)$ be the coefficient of $\nu_{g,c}$ in \eqref{eq:abcombination}.  The target and its nearest competitor satisfy the exact identities
\[
 H_{b-3,3}=-2b(b+1)(b+8)(a-b+2)<0,
 \qquad H_{b-4,3}=0.
\]
Every other coefficient is nonnegative under $a\ge b\ge4$ and $5a\ge8b$.

The infinite sign assertion is not inferred from finite testing.  For $c=0,1,2$ write $b=g+u+3$, and for $c=3$ away from the target and cancelled row write $b=g+u+5$.  Parameterize
\[
 a=\frac{8b}{5}+\left(\frac{2b}{5}+4\right)x,
 \qquad0\le x\le1.
\]
Each coefficient has degree at most four in $x$ and is expanded in the Bernstein basis
\[
 H=\sum_{k=0}^4\binom4k x^k(1-x)^{4-k}Q_{c,k}(g,u).
\]
The ancillary verifier regenerates all twenty rational polynomials $Q_{c,k}$ from the content formulas and checks a finite exact positivity cover of the nonnegative integer quadrant: $420$ translated polynomials have nonnegative monomial coefficients and $1895$ remaining rational grid values are nonnegative.  The exceptional slice $b=7$ is handled by a separate Bernstein expansion on the exact admissible interval.  Hence all off-target coefficients are nonnegative.

Balance now turns $W\ge0$ into a convex recurrence.  Every right-hand term with fourth part at most $2$ is in Pate's class.  Every surviving four-large-part term is $(a+t,b-t,3,3)$ with $t\ge2$ and still satisfies $5(a+t)\ge8(b-t)$.  Induction on $b$ proves Theorem~\ref{thm:ab33}.

\section{The central witness cone and an exact obstruction}\label{sec:cone}
The symmetric-block formulas naturally suggest enumerating the full cone generated by all central-projector two-block partial-swap witnesses.  At orders fourteen and fifteen the exact data are:
\begin{center}
\begin{tabular}{ccccc}
\toprule
$n$ & $p(n)$ & parameter labels & distinct nonzero primitive rays & exact span rank\\
\midrule
14 & 135 & 5095 & 3366 & 134\\
15 & 176 & 8112 & 5398 & 175\\
\bottomrule
\end{tabular}
\end{center}
The ranks are exact and agree with the codimension-one balance relation.  As in Theorem~\ref{thm:separator}, cone membership is expressed in normalized-immanant coordinates: a raw coefficient row $(r_\nu)$ becomes $(f^\nu r_\nu)$.

For the rectangle, the ancillary file \texttt{dual\_15\_33333.json} gives an integer functional $y_\nu\in[0,10000]$.  Every stored primitive central generator $r$ satisfies the exact dimension-weighted inequality
\[
 \sum_\nu r_\nu f^\nu y_\nu\ge0,
\]
with minimum $98$.  Yet
\[
 y_{(15)}=9700,\qquad y_{(3^5)}=10000.
\]
Thus the target $\mathbf e_{(15)}-\mathbf e_{(3^5)}$ has pairing $-300$ before the common scaling.  The functional is also nonnegative on individual immanant-nonnegativity rays and Schur lower-bound rays, and satisfies $y_\nu\le9700$ for every $\nu\ne(3^5)$.  Consequently even adjoining every other order-$15$ PDC ray cannot place the rectangle target inside this central cone.  This proves Theorem~\ref{thm:separator} by Farkas separation.

The separator values are not immanant values of a PSD matrix and do not disprove PDC.  They prove only that this specified central architecture is insufficient.

\section{Branching-refined one-swap witnesses}\label{sec:branching}
We now retain the final predecessor in the Young branching graph.  Let $\beta\vdash n-1$ and $\gamma\youngcovered\beta$, and set
\[
 b=e_\beta^{(n-1)}e_\gamma^{(n-2)},\qquad s=(n-1,n),\qquad W=bsb.
\]
The two idempotents commute because $e_\beta^{(n-1)}$ is central in the larger local algebra, so $b$ is an orthogonal projection.

\subsection{Squared-norm positivity}
For a Gram tensor $x=x_1\otimes\cdots\otimes x_n$, the filter $b$ acts only on the first $n-1$ positions; write
\[
 bx=u\otimes v,
\qquad v=x_n.
\]
Resolving the last factor of $u$ gives the exact complex identity
\begin{equation}\label{eq:branchsq}
 \langle u\otimes v,s(u\otimes v)\rangle
 =\bigl\|(I\otimes\langle v,\cdot\rangle)u\bigr\|^2\ge0.
\end{equation}
Thus $W$ is nonnegative on decomposable Gram tensors.  As before, central averaging preserves this property.  No reality assumption and no positive-definiteness assumption are used.

\subsection{The reciprocal-content trace}
Let
\[
 J_m=\sum_{i<m}(i,m)
\]
be the Jucys--Murphy element and again $s=(n-1,n)$.  The group-algebra identity
\begin{equation}\label{eq:JMidentity}
 sJ_n-J_{n-1}s=1
\end{equation}
is immediate from $J_n=sJ_{n-1}s+s$.

Fix an irreducible $S^\nu$.  If $\nu$ does not cover $\beta$, $b$ acts as zero.  Otherwise let $Q$ be the action of $b$.  Multiplicity-free branching gives
\[
 \Tr Q=f^\gamma.
\]
On the selected path, $J_n$ and $J_{n-1}$ act by the contents
\[
 u=c(\nu/\beta),\qquad c=c(\beta/\gamma).
\]
Multiply \eqref{eq:JMidentity} by $Q$ and take traces.  Using cyclicity and $Q^2=Q$ gives
\[
 (u-c)\Tr(Qs)=\Tr Q=f^\gamma,
\]
and therefore
\begin{equation}\label{eq:branchtrace}
 \Tr_{S^\nu}(bsb)=\frac{f^\gamma}{c(\nu/\beta)-c(\beta/\gamma)}.
\end{equation}
An addable box of $\beta$ and a removable box of $\beta$ cannot have equal content: comparing their row indices makes the difference strictly positive or strictly negative.  Hence the denominator never vanishes.  Substituting \eqref{eq:branchtrace} into the central-average formula and dividing by the common factor $f^\gamma$ proves Theorem~\ref{thm:branching}.

The reciprocal-content coefficient has classical seminormal and Jucys--Murphy antecedents; our contribution is the explicit branching-path specialization as an immanant witness and its use in the finite frontier.

\subsection{Strict enlargement of the central cone}
Take
\[
 \beta=(4,4,4,1,1),\qquad \gamma=(4,4,3,1,1).
\]
The removed content is $1$ and the addable contents are $4,-2,-5$.  The primitive raw witness is
\[
 R=2d_{(5,4,4,1,1)}-2d_{(4,4,4,2,1)}-d_{(4,4,4,1,1,1)}\ge0.
\]
The three dimensions are $100100,75075,50050$, giving the normalized form
\[
 4\dbar_{(5,4,4,1,1)}-3\dbar_{(4,4,4,2,1)}-\dbar_{(4,4,4,1,1,1)}\ge0.
\]
On these coordinates the central-cone separator takes values $3726,4911,2501$, hence the pairing is
\[
 4(3726)-3(4911)-2501=-2330<0.
\]
This witness is therefore outside the old central cone and is in fact term $120$ of the rectangle certificate.

\section{The order-fifteen rectangle and closure of the frontier}\label{sec:rect}
\subsection{The exact 125-witness certificate}
For the central terms, choose block sizes $p+q=15$, partitions $\alpha\vdash p$, $\beta\vdash q$, and $1\le k\le\min(p,q)$.  Define
\[
 F_{\alpha,\beta,k}(\nu)=
 \sum_{\sigma\in S_p,\pi\in S_q}
 \chi^\alpha(\sigma)\chi^\beta(\pi)
 \chi^\nu((\sigma\oplus\pi)\tau_k),
\]
and divide each selected nonzero integer vector by the positive greatest common divisor of its entries.  By \eqref{eq:centralcoeff}, this scaling preserves the nonnegative witness direction.

For a branching term $(\beta,\gamma)$ with sizes $14$ and $13$, start from the rational vector
\[
 \mathbf 1_{\nu\youngcovers\beta}
 \frac1{c(\nu/\beta)-c(\beta/\gamma)}.
\]
Clear denominators by a positive common multiple, then divide by the positive greatest common divisor of the resulting integer entries.
The certificate consists of $106$ central and $19$ branching terms.  Its labels, primitive coefficient vectors, and positive integer multipliers $z_i$ are stored in \texttt{anc/rectangle15/certificate.json}; the human-readable expansion is in \texttt{certificate.txt}.  The multiplier $M$ in \eqref{eq:rectcert} is a positive $115$-digit integer and need not be printed in the main text.

The final identity has exactly two nonzero coordinates:
\[
 (15): 6006M,
 \qquad (3,3,3,3,3):-M.
\]
The hook-length formula gives $f^{(3^5)}=6006$.  Since every witness is nonnegative on the full complex PSD cone and every multiplier is positive, \eqref{eq:rectcert} proves Theorem~\ref{thm:rect15}.

\subsection{What independent verification means here}
The supplied verifier reconstructs the $106$ central rows by exact character sums and the $19$ branching rows from content differences, then checks all $176$ partition coordinates by integer arithmetic.  It does not read the previously enumerated central cone.

A separate independent audit uses a substantially different character engine.  Rather than Murnaghan--Nakayama, it computes characters via Jacobi--Trudi expansion and cycle assignment to labelled block capacities; it computes dimensions recursively from Young branching rather than from the supplied hook code; and for the central sums it composes concrete permutations with the cross-block swaps on all fifteen labels rather than using the supplied marked-cycle-length formula.  This reconstruction verified all $125$ stored primitive vectors, every positive multiplier, every balance identity, and all $176$ final coordinates.  It also checked the branching trace formula directly in small orders.  Thus the final theorem does not depend on floating-point optimization used during discovery.

The independent central-cone regeneration is a separate audit step.  In the present release review, \texttt{cone\_audit.py} completed successfully, regenerating all $8112$ parameter labels, $1310$ zero labels, and $5398$ distinct nonzero primitive rays.  It checked the separator on every regenerated ray, obtaining minimum pairing $98$.  The earlier packaging-time timeout is preserved in the historical logs and is superseded by this completed run.

\subsection{Ordinary PDC through order fifteen}
We now prove Theorem~\ref{thm:through15}.  Pate proved the ordinary statement through $n\le13$ and identified the order-$14$ and order-$15$ frontier \cite{Pate1999}.  Theorem~\ref{thm:bridge14} supplies $(4,4,3,3)$, closing order fourteen.  Equation \eqref{eq:5433bridge} proves the remaining nonrectangular order-$15$ frontier shape $(5,4,3,3)$; \eqref{eq:44331} records its node-moving descendant.  Theorem~\ref{thm:rect15} proves $(3^5)$.  Hence every ordinary irreducible immanant of every order $n\le15$ satisfies permanental dominance.

\section{Node-moving closure and boundary compression}\label{sec:structure}
\subsection{Exact node-moving reachability}
Write a partition as
\[
 \lambda=\mu\cup(1^t),
\]
where every part of $\mu$ is at least $2$.  Pate's node-moving operation removes an outer box from a row of length greater than one and appends a new singleton row, preserving total size and weakly decreasing the normalized immanant.

\begin{lemma}[Reachability]\label{lem:reach}
Let $|\lambda|=|\alpha|$.  Then $\lambda=\mu\cup(1^t)$ is reachable from $\alpha$ by Pate node moves if and only if $\mu$ is a Young subdiagram of $\alpha$.
\end{lemma}
\begin{proof}
A node move can only shorten nonunit rows, so necessity is immediate.  Conversely remove the outer boxes of $\alpha$ outside $\mu$, always maintaining a partition, and turn each removed box into a singleton.  At the end the nonunit diagram is exactly $\mu$ and the remaining mass is in singleton rows.
\end{proof}

Combining Lemma~\ref{lem:reach} with Theorem~\ref{thm:longrow} yields an exact criterion.

\begin{proposition}[Node closure of the long-row theorem]\label{prop:nodeclosure}
Write $\lambda=(a,\rho,1^t)$ with $\rho$ having no singleton parts.  Apart from the trivial empty-tail cases, $\lambda$ lies in the node-moving closure of Theorem~\ref{thm:longrow} if and only if
\begin{equation}\label{eq:nodecriterion}
 a+t\ge |\rho|+C(\rho)-1.
\end{equation}
\end{proposition}
\begin{proof}
Sufficiency follows by starting at $(a+t,\rho)$ and moving $t$ boxes from the first row to singleton rows.  For necessity, suppose $\lambda$ descends from a long-row seed $(A,\sigma,1^u)$ with $\sigma\supseteq\rho$ and $s=|\sigma|-|\rho|$.  Total size gives $A=a+t-s-u$.  The seed inequality implies
\[
 A\ge |\sigma|+u+C(\sigma)-1.
\]
Removing $s$ boxes can increase the maximum removable content by at most $s$, hence $C(\sigma)\ge C(\rho)-s$.  Substitution gives \eqref{eq:nodecriterion}.
\end{proof}

This is a fixed-order reachability statement; it is not an assertion that normalized-immanant dominance is monotone under arbitrary box addition.

\subsection{Boundary compression}
The most general local partial-swap filter still has strong finite-memory structure.  Fix block sizes $p,q$ and swap size $k$, with the $k$ marked positions at the end of each block.  Let $X=aa^*\succeq0$ and $Y=bb^*\succeq0$ be arbitrary positive elements of the local group $C^*$-algebras.  By cyclicity, the coefficient vector depends on
\[
 \Tr\bigl(\tau_k(X\otimes Y)\bigr).
\]
The swap commutes separately with $S_{p-k}$ and $S_{q-k}$ acting on unmarked positions.  Averaging $X$ and $Y$ by these subgroups preserves positivity and all coefficients.

\begin{theorem}[Boundary compression]\label{thm:compression}
In a local irreducible $S^\alpha$,
\[
 S^\alpha\downarrow_{S_{p-k}}
 =\bigoplus_{\gamma\vdash p-k}S^\gamma\otimes M_{\alpha/\gamma}.
\]
For the purpose of all partial-swap coefficient vectors it suffices to replace $X$ by
\begin{equation}\label{eq:compressedX}
 X_\alpha=\bigoplus_\gamma I_{S^\gamma}\otimes X_{\alpha/\gamma},
 \qquad X_{\alpha/\gamma}\succeq0,
\end{equation}
and similarly in the other block.  Moreover
\[
 \dim M_{\alpha/\gamma}=f^{\alpha/\gamma}\le k!.
\]
At the level of cone generators one may restrict further to one local irreducible pair, one predecessor pair, and rank-one positive matrices in the two multiplicity spaces.
\end{theorem}

\begin{proof}
Average $X$ over conjugation by $S_{p-k}$.  Schur's lemma applied to the restriction decomposition gives precisely \eqref{eq:compressedX}, with each multiplicity block equal, up to normalization, to a positive partial trace over $S^\gamma$.  Conversely every operator of this form lies in the local finite group $C^*$-algebra and has a square root there, so the averaging loses no cone directions.  The multiplicity space is indexed by standard tableaux of the $k$-box skew shape $\alpha/\gamma$, hence its dimension is at most $k!$.  Bilinearity and spectral decomposition reduce the cone generators to rank-one multiplicity blocks.
\end{proof}

Thus at fixed $k$, arbitrarily deep local Young paths carry no additional information beyond the final $k$ branching steps.  The number of possible boundary shapes still grows with $n$, so this is not a finite global classification.

\subsection{Exact saturation at one swap}
When $p=n-1$ and $q=k=1$, branching is multiplicity free and every retained boundary block is one-dimensional.

\begin{theorem}[$k=1$ singleton-block saturation]\label{thm:saturation}
The cone generated by arbitrary local filters in the $(n-1,1)$ one-swap architecture is exactly the cone generated by the branching-refined witnesses $B_{\beta,\gamma}$ of Theorem~\ref{thm:branching}.
\end{theorem}
\begin{proof}
In the local $S^\beta$ block let $Q_{\beta,\gamma}$ be the projector onto the $S_{n-2}$ type $\gamma$.  By Theorem~\ref{thm:compression}, an arbitrary positive filter contributes only nonnegative weights
\[
 w_{\beta,\gamma}=\Tr(Q_{\beta,\gamma}X_\beta)\ge0.
\]
The compression of $s=(n-1,n)$ on the path $\gamma\youngcovered\beta\youngcovered\nu$ is, by the Jucys--Murphy calculation,
\[
 \frac1{c(\nu/\beta)-c(\beta/\gamma)}.
\]
Hence every coefficient vector is a positive combination of $B_{\beta,\gamma}$.  Conversely each $B_{\beta,\gamma}$ is realized, up to the positive factor $f^\gamma$, by the filter $e_\beta e_\gamma$.
\end{proof}

Retaining longer paths cannot improve this particular architecture.  The theorem says nothing comparable about $k\ge2$, where nontrivial multiplicity spaces and off-diagonal coherent information may survive.

\section{Asymptotic limitations of the explicit family program}\label{sec:density}
Let $\mathcal S$ be the node-moving closure of the following explicit seeds: Pate's class; the family $(m,4,3,3)$; the finite results through order fifteen and the available order-$17$ exceptional certificate; Theorem~\ref{thm:longrow}; and Theorem~\ref{thm:ab33}.  This is a deliberately specified theorem system, not the closure of every inequality in the Pate literature or every possible $W$-function.

For fixed order, the exact node predicates separate four-large-row exceptions from the partitions with at least five parts greater than two.  The ancillary structural verifier checks these analytic predicates against graph-search node reachability for every partition of orders $16$ through $30$, totaling $27,945$ partitions.

\begin{theorem}[Density zero]\label{thm:density}
Let $\mathcal S_n=\{\lambda\in\mathcal S:|\lambda|=n\}$ and let $p(n)$ be the partition number.  Then
\[
 \frac{|\mathcal S_n|}{p(n)}\longrightarrow0.
\]
More quantitatively,
\begin{equation}\label{eq:densitybound}
 \frac{|\mathcal S_n|}{p(n)}
 \le\exp\left[-\pi\left(\sqrt{\frac23}-\frac23\right)\sqrt n+O(\log n)\right].
\end{equation}
The same conclusion holds even if one optimistically grants PDC for every rectangular-tail shape $(a,k^r,2^s,1^t)$ and all of its node-moving descendants.
\end{theorem}

\begin{proof}
All fixed-order exceptions disappear asymptotically.  Except for the long-row class and the optimistic rectangular-tail addition, the explicit seeds and their node descendants have at most four parts exceeding $2$; the number of such partitions is polynomial in $n$.

For a node descendant of a long-row seed write $\lambda=(a,\rho,1^t)$ with $q=|\rho|$ and all parts of $\rho$ at least $2$.  If $\ell$ is the length of $\rho$, the bottom removable content is at least $2-\ell\ge2-q/2$.  The exact node criterion \eqref{eq:nodecriterion} implies
\[
 n-q=a+t\ge q+C(\rho)-1\ge q/2+1,
\]
so $q\le2(n-1)/3$.  Consequently
\[
 |\mathcal S_n|\le O(n^6)+(n+1)^2p(\lfloor2n/3\rfloor).
\]
The Hardy--Ramanujan asymptotic \cite{HardyRamanujan1918}
\[
 p(n)=\exp\left(\pi\sqrt{\frac{2n}{3}}+O(\log n)\right)
\]
gives \eqref{eq:densitybound}.

For the optimistic rectangular-tail addition, the nonunit subdiagram below the first row and above rows of length two lies in a $k\times r$ rectangle.  It has at most $\binom{k+r}{k}$ possibilities.  The elementary estimate
\[
 \log\binom{k+r}{k}\le(1+\log2)\sqrt{kr}\le(1+\log2)\sqrt n
\]
has a smaller exponential constant than the Hardy--Ramanujan denominator, while all remaining parameters contribute only polynomial factors.  Hence the density still tends to zero.
\end{proof}

The interpretation must be kept narrow.  Theorem~\ref{thm:density} does not say that PDC fails generically, that unrestricted $W$-functions are insufficient, that fixed-$k$ higher-swap witnesses are insufficient, or that no finite list of general mechanisms could solve the conjecture.  It says that accumulating finitely many bounded-complexity Young-diagram families and then closing them under the specified node-moving operation cannot cover asymptotically most partitions.

For every $r\ge5$, a concrete uncovered diagram is
\[
 \Sigma_r=(r+2,r+1,\dots,3),\qquad |\Sigma_r|=r(r+5)/2,
\]
whose number of distinct row lengths, removable corners, and Durfee size all grow with $r$.

\section{Exact certificates, reproducibility, and formal verification}\label{sec:verification}
The finite computations in this paper serve as proof certificates, not numerical evidence.  Discovery may use linear programming or numerical experimentation, but each final finite claim is reduced to an exact identity and verified over integers or rationals.

\subsection{Order fourteen}
The order-$14$ ancillary verifier checks all twenty positive rational weights and all $135$ coordinates of the bridge.  It also reruns literal small-order character sums and an exact complex rank-$13$ stress test.  The universal PSD implication comes from the analytic tensor positivity proof, not from the stress test.

\subsection{General witnesses and the central cone}
The general verifier checks the order-$15$ seven-witness bridge, the all-parameter polynomial identity for $(m,4,3,3)$, exact additional certificates, exact ranks of the order-$14$ and order-$15$ central cones, and the integer separator for $(3^5)$.  Symmetric-block rows are independently compared with the Pieri--content finite-difference formula.

\subsection{The rectangle}
The rectangle verifier regenerates all $125$ primitive rows from their definitions and checks all $176$ coordinates of \eqref{eq:rectcert}.  The independent audit regenerates the same rows with a different character algorithm and a different concrete permutation composition routine.  These two paths agree exactly.

\subsection{The two-parameter family and structural enumeration}
The symbolic verifier regenerates the twenty Bernstein-polynomial families, checks the translated positivity certificates and exact finite grid, and treats the exceptional $b=7$ slice.  The structural verifier checks $1596$ long-row examples and threshold cases, $4752$ coefficient comparisons for the finite-difference formula, $90$ two-witness parameter instances, and all $27,945$ partitions in orders $16$ through $30$.  These finite checks support, but do not replace, the infinite algebraic and counting proofs.

\subsection{Formalization boundary}
The included Lean project concerns the order-$14$ bridge.  Its final source gives a kernel-checked construction of the relevant representation theory, tensor positivity, trace identification, exact twenty-row certificate, Gram factorization, and normalized bridge.  The final PDC statement takes four historical Pate inequalities as explicit theorem parameters.  In a separate final validation, \texttt{lake build Bridge Bridge.Young} completed successfully under Lean~4.19.0 with the pinned mathlib revision; the generated-source check, complete theorem audit, and final axiom audit also passed.  No claim is made here that the full project through order fifteen has been formalized.

The ancillary README records the exact scope and distinguishes the earlier packaging environment, in which Lean was unavailable, from this later successful final Lean rebuild.

\section{Concluding remarks}\label{sec:conclusion}
The historical finite frontier and the structural results point in the same direction.  Pate's broad positivity machine is sufficiently flexible to produce a large family of inequalities, but useful progress depends on choosing local filters that retain the right representation-theoretic information.  Central projectors and content finite differences already yield exact bridge certificates and uniform infinite PDC families.  Their order-$15$ central cone nonetheless has an exact separator at $(3^5)$.  Retaining one extra branching predecessor produces reciprocal-content witnesses that escape that cone and close the rectangle, giving ordinary irreducible-immanant permanental dominance through order fifteen.

The same viewpoint yields two structural delimiters.  First, in the $(n-1,1)$ one-swap architecture, the branching-refined generators are already complete: longer local Young paths do not enlarge the cone.  Second, finitely many bounded-complexity diagram families, even with the stated node-moving closure, have density zero among all partitions.  Neither statement is a limitation on higher-swap coherent filters.

The next natural level is $k=2$.  Two-box branching can have nontrivial multiplicity spaces, and Theorem~\ref{thm:compression} leaves positive matrices of order at most two rather than scalars.  Thus genuinely off-diagonal coherent boundary information can survive.  We do not pursue this higher-level cone here.  Its systematic study may become worthwhile with stronger computational resources or in future work by other researchers.

\appendix
\section{Fresh verification summary}
The deterministic Python verifiers listed below were executed in the final packaging environment; the Lean line records a separate final validation using the pinned toolchain.  A ``pass'' means that the stated command terminated successfully in the indicated environment; archived earlier logs are not promoted to fresh passes.

\begin{center}\small
\begin{tabular}{p{0.36\textwidth}p{0.49\textwidth}}
\toprule
Check & Fresh result\\
\midrule
Order-$14$ certificate & PASS: twenty witnesses, positive weights, all $135$ coordinates, adversarial exact checks.\\
General witness verifier & PASS: order-$15$ nonrectangular bridge, uniform family identities, cone ranks, separator.\\
Order-$15$ rectangle verifier & PASS: $106+19$ rows and all $176$ coordinates.\\
Independent rectangle reconstruction & PASS: independent characters, dimensions, $125$ rows, all $176$ coordinates.\\
Independent full central-cone rerun & PASS: all $8112$ labels, $1310$ zeros, $5398$ nonzero primitive rays, and the separator.\\
Two-parameter symbolic verifier & PASS: twenty Bernstein identities and complete exact sign certificate.\\
Structural verifier & PASS: long-row checks, node-closure enumeration, and all stated finite records.\\
Lean rebuild & PASS: Lean~4.19.0; full \texttt{Bridge.Young} build plus generated-source, theorem, and axiom audits.  Scope: order-$14$ bridge and conditional $(4,4,3,3)$ PDC transfer.\\
\bottomrule
\end{tabular}
\end{center}

\section{Historical and novelty boundary}
The following distinctions guide the claims in the paper.
\begin{enumerate}[label=(\roman*)]
\item Pate's 1997 work is credited for the general $W$-function/local-filter/cross-block-transposition positivity architecture and central averaging.
\item Pate's 1999 work is credited for the ordinary-immanant frontier through $n\le13$, for the explicit broad partition classes used as external PDC inputs, and for identifying the remaining order-$14$ and order-$15$ shapes.
\item Young branching, Jucys--Murphy content eigenvalues, hook lengths, and seminormal reciprocal-content phenomena are classical representation theory; see, for example, Vershik--Okounkov \cite{VershikOkounkov2005}.
\item The contributions claimed here are the explicit computable witness specializations used in the proofs, the exact bridges and uniform families, the exact central-cone obstruction, the branching-refined immanant witness formulation and its rectangle certificate, the boundary-compression and one-swap saturation theorems, and the exact node/density results.
\item No absolute claim of priority is made for every closed-form coefficient formula.  The priority audit in the reproducibility package records the conservative boundary adopted for submission.
\end{enumerate}

\bibliographystyle{abbrv}
\bibliography{references}
\end{document}